\documentclass[pdftex]{amsart}
\usepackage{amsmath,amsthm,amssymb}
\usepackage{tikz}
\usetikzlibrary{arrows.meta,calc,decorations.markings,fit,math,cd}
\usepackage[skins,raster,breakable]{tcolorbox}
\usepackage{mathtools}
\mathtoolsset{showonlyrefs}
\usepackage[style=alphabetic,backend=biber]{biblatex}
\numberwithin{equation}{section}

\renewcommand{\P}{\mathbb{P}}
\newcommand{\Q}{\mathbb{Q}}
\newcommand{\R}{\mathbb{R}}
\newcommand{\can}{\hat{h}_\phi}

\newtheorem{lem}{Lemma}[section]

\newtheorem{thm}[lem]{Theorem}
\theoremstyle{definition}
\newtheorem{rem}[lem]{Remark}
\newtheorem{dfn}[lem]{Definition}

\title[Height Gaps for Polynomial Maps]{On the Canonical Height Gap for Polynomial Maps and Portraits of Preperiodic Points of Polynomials with Height $0$}
\author{Haruki Imamura}
\date{\today}

\begin{document}
	
	\begin{abstract}
		We study the difference between the canonical height and the naive height for polynomial maps on $\mathbb{P}^{1}$. While explicit upper bounds on this height gap generally depend on the degree $d$ for rational maps, we establish a refined bound for polynomial maps that is essentially independent of $d$. As an application, we determine the set of rational preperiodic points for polynomial maps defined over $\mathbb{Q}$ of height $0$ and classify their realizable portraits.
	\end{abstract}
	
	\maketitle
	\setcounter{tocdepth}{1}
	\tableofcontents
	
	\section{Introduction}
	
	For a non-isomorphic surjective morphism $\phi: \mathbb{P}^1 \to \mathbb{P}^1$ of projective line defined over a number field, the canonical height $\hat{h}_\phi$ is a height function associated with $\phi$, reflecting the dynamical properties of the map. It is classically known that the difference between the naive height $h$, defined in Definition \ref{naive_height}, and the canonical height $\hat{h}_\phi$ is bounded. This difference plays a crucial role in understanding the distribution of preperiodic points and other arithmetic properties.
	In this paper, we focus on polynomial maps on $\mathbb{P}^1$ and establish a refined upper bound on the height gap, improving upon the result of \cite{yap}. As an application of this bound, we determine the set of rational preperiodic points for polynomial maps of height 0 and classify their realizable portraits.
	
	Let us first recall the definition of the canonical height.
	
	\begin{dfn}[Canonical Height]
		Let $\phi \colon \P^1 \to \P^1$ be a rational map of degree $d \geq 2$ defined over a number field $K$. We define the \emph{canonical height} $\hat{h}_\phi \colon \P^1(\overline{\Q}) \to \R$ as
		\begin{equation}
			\can(P) \coloneq \lim_{n \to \infty} \frac{h\big( \phi^n(P) \big)}{d^n}.
		\end{equation}
		This limit exists (see, for example, \cite[Theorem 3.20]{sil}). Here, $h$ is the naive height function on $\mathbb{P}^1$ (see Definition \ref{naive_height}).
	\end{dfn}
	
	For rational maps, we define their height as follows:
	
	\begin{dfn}
		Let $\phi=(f_0:f_1): \mathbb{P}^1 \to \mathbb{P}^1$ be a rational map defined over a number field $K$ of degree $d \geq 2$. Let each $f_i$ be written as
		\begin{align*}
			f_i(X_0, X_1) = \sum_{j=0}^d a_{ij} X_0^{d-j} X_1^{j}.
		\end{align*}
		Then we define the height $h(\phi)$ of $\phi$ to be the height of the point whose coordinates are all the coefficients $a_{ij}$ of the polynomials $f_0$ and $f_1$.
		Note that this definition is independent of the presentation of each $f_i$ because $h$ does not depend on homogeneous coordinates.
	\end{dfn}
	
	In this paper, we say that a rational map $\phi \colon \mathbb{P}^1 \to \mathbb{P}^1$ of degree $d \geq 2$ is a \emph{polynomial map} if we can take $f_1(X_0, X_1) = X_1^d$. The following is our main theorem:
	
	\begin{thm}
		\label{main thm}
		Let $\phi \colon \P^1 \to \P^1$ be a polynomial map defined over a number field $K$ of degree $d \geq 2$. Then
		\begin{equation}
			|\hat{h}_\phi(P) - h(P)| \leq \frac{d}{d - 1} \left( \log2 + \frac{\log d}{d} + h(\phi) \right)
		\end{equation}
		for all $P \in \P^1(\overline{K})$.
	\end{thm}
	
	Yap refined in \cite{yap} the results of \cite{hutz} and \cite{ingram} by working on $\mathbb{P}^1$. Theorem \ref{main thm} can be viewed as a further refinement of \cite[Proposition 3.2]{yap} for the specific case of polynomial maps. Crucially, while the bound for general rational maps on $\mathbb{P}^1$ in \cite[Proposition 3.2]{yap} is of the form $O(d h(\phi) + d \log d)$, our result for polynomial maps gives a bound of the form $O(h(\phi) + 1)$.
	
	We apply Theorem \ref{main thm} to classify preperiodic points of polynomials of height $0$.
	
	\begin{thm}
		\label{-2 to 2}
		Let $\phi$ be a polynomial map of height $0$ defined over $\mathbb{Q}$. If a rational number $\alpha$ is preperiodic for $\phi$, then $\alpha$ is either $0$, $\pm 1$, or $\pm 2$.
	\end{thm}
	
	We also determine which portraits can be realized.
	
	\begin{thm}
		\label{portrait thm}
		Let $\phi$ be a polynomial map of height $0$ defined over $\mathbb{Q}$.
		The portrait of preperiodic points for $\phi$ is one of those listed below, where we exclude the point at infinity. Conversely, each portrait in the list is realized by some such map.
		\begin{tcbitemize}[raster columns=2, raster equal height,halign=center,valign=center,sharp corners,size=small,colframe=black,colback=white,colbacktitle=red!50!white,raster column skip=0pt,raster row skip=0pt]
			\tcbitem 
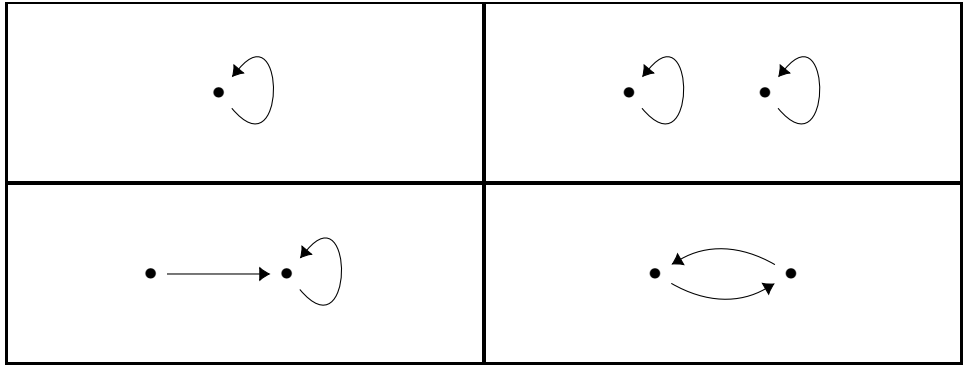
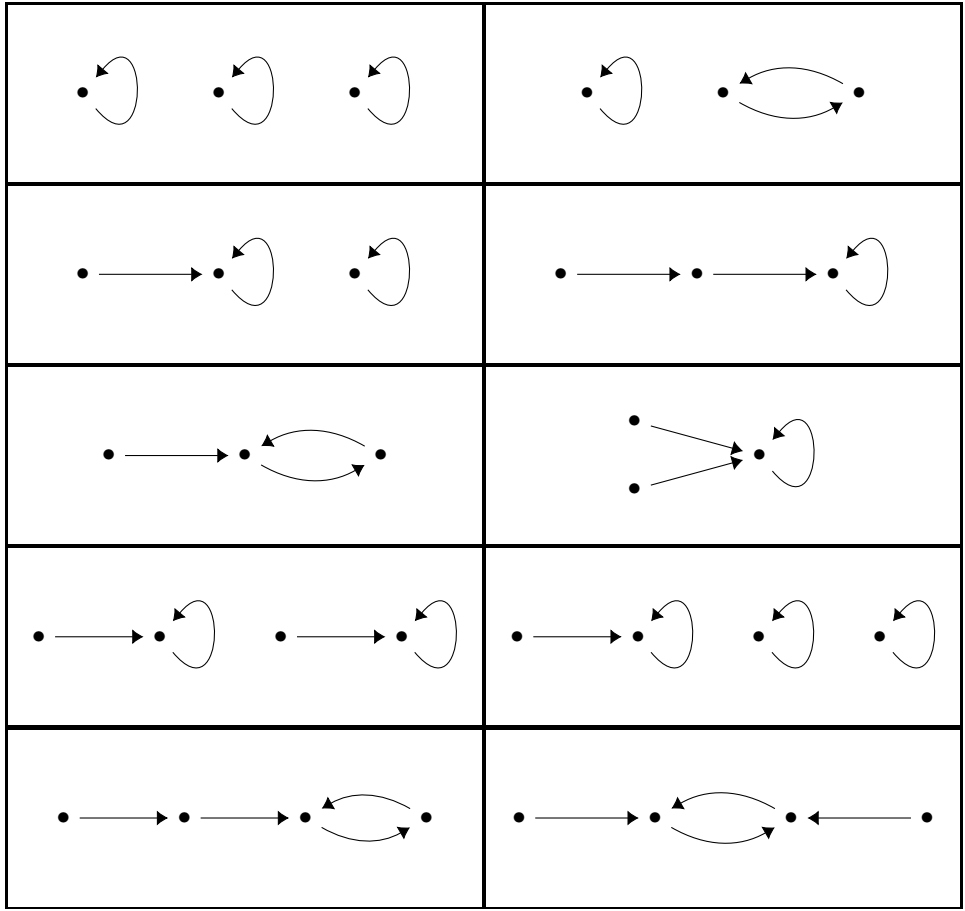
\begin{tikzpicture}[scale=.9]
				\tikzset{vertex/.style = {}}
				\tikzset{every loop/.style={min distance=10mm,in=45,out=-45,->}}
				\tikzset{edge/.style={decoration={markings,mark=at position 1 with %
							{\arrow[scale=1.5,>=stealth]{>}}},postaction={decorate}}}
				\node[vertex] (1) at (0, 0) {$\bullet$};
				\draw[-{Latex[length=1.5mm,width=2mm]}] (1) to[out=310, in=50, looseness=7] (1);
			\end{tikzpicture}
			\tcbitem \begin{tikzpicture}[scale=.9]
				\tikzset{vertex/.style = {}}
				\tikzset{every loop/.style={min distance=10mm,in=45,out=-45,->}}
				\tikzset{edge/.style={decoration={markings,mark=at position 1 with %
							{\arrow[scale=1.5,>=stealth]{>}}},postaction={decorate}}}
				\node[vertex] (1) at (0, 0) {$\bullet$};
				\node[vertex] (2) at (2, 0) {$\bullet$};
				\draw[-{Latex[length=1.5mm,width=2mm]}] (1) to[out=310, in=50, looseness=7] (1);
				\draw[-{Latex[length=1.5mm,width=2mm]}] (2) to[out=310, in=50, looseness=7] (2);
			\end{tikzpicture}
			\tcbitem \begin{tikzpicture}[scale=.9]
				\tikzset{vertex/.style = {}}
				\tikzset{every loop/.style={min distance=10mm,in=45,out=-45,->}}
				\tikzset{edge/.style={decoration={markings,mark=at position 1 with %
							{\arrow[scale=1.5,>=stealth]{>}}},postaction={decorate}}}
				\node[vertex] (1) at (0, 0) {$\bullet$};
				\node[vertex] (2) at (2, 0) {$\bullet$};
				\draw[-{Latex[length=1.5mm,width=2mm]}] (1) to (2);
				\draw[-{Latex[length=1.5mm,width=2mm]}] (2) to[out=310, in=50, looseness=7] (2);
			\end{tikzpicture}
			\tcbitem \begin{tikzpicture}[scale=.9]
				\tikzset{vertex/.style = {}}
				\tikzset{every loop/.style={min distance=10mm,in=45,out=-45,->}}
				\tikzset{edge/.style={decoration={markings,mark=at position 1 with %
							{\arrow[scale=1.5,>=stealth]{>}}},postaction={decorate}}}
				\node[vertex] (1) at (0, 0) {$\bullet$};
				\node[vertex] (2) at (2, 0) {$\bullet$};
				\draw[-{Latex[length=1.5mm,width=2mm]}] (1) to[bend right=30] (2);
				\draw[-{Latex[length=1.5mm,width=2mm]}] (2) to[bend right=30] (1);
			\end{tikzpicture}
			\tcbitem \begin{tikzpicture}[scale=.9]
				\tikzset{vertex/.style = {}}
				\tikzset{every loop/.style={min distance=10mm,in=45,out=-45,->}}
				\tikzset{edge/.style={decoration={markings,mark=at position 1 with %
							{\arrow[scale=1.5,>=stealth]{>}}},postaction={decorate}}}
				\node[vertex] (1) at (0, 0) {$\bullet$};
				\node[vertex] (2) at (2, 0) {$\bullet$};
				\node[vertex] (3) at (4, 0) {$\bullet$};
				\draw[-{Latex[length=1.5mm,width=2mm]}] (1) to[out=310, in=50, looseness=7] (1);
				\draw[-{Latex[length=1.5mm,width=2mm]}] (2) to[out=310, in=50, looseness=7] (2);
				\draw[-{Latex[length=1.5mm,width=2mm]}] (3) to[out=310, in=50, looseness=7] (3);
			\end{tikzpicture}
			\tcbitem \begin{tikzpicture}[scale=.9]
				\tikzset{vertex/.style = {}}
				\tikzset{every loop/.style={min distance=10mm,in=45,out=-45,->}}
				\tikzset{edge/.style={decoration={markings,mark=at position 1 with %
							{\arrow[scale=1.5,>=stealth]{>}}},postaction={decorate}}}
				\node[vertex] (1) at (0, 0) {$\bullet$};
				\node[vertex] (2) at (2, 0) {$\bullet$};
				\node[vertex] (3) at (4, 0) {$\bullet$};
				\draw[-{Latex[length=1.5mm,width=2mm]}] (1) to[out=310, in=50, looseness=7] (1);
				\draw[-{Latex[length=1.5mm,width=2mm]}] (2) to[bend right=30] (3);
				\draw[-{Latex[length=1.5mm,width=2mm]}] (3) to[bend right=30] (2);
			\end{tikzpicture}
			\tcbitem \begin{tikzpicture}[scale=.9]
				\tikzset{vertex/.style = {}}
				\tikzset{every loop/.style={min distance=10mm,in=45,out=-45,->}}
				\tikzset{edge/.style={decoration={markings,mark=at position 1 with %
							{\arrow[scale=1.5,>=stealth]{>}}},postaction={decorate}}}
				\node[vertex] (1) at (0, 0) {$\bullet$};
				\node[vertex] (2) at (2, 0) {$\bullet$};
				\node[vertex] (3) at (4, 0) {$\bullet$};
				\draw[-{Latex[length=1.5mm,width=2mm]}] (1) to (2);
				\draw[-{Latex[length=1.5mm,width=2mm]}] (2) to[out=310, in=50, looseness=7] (2);
				\draw[-{Latex[length=1.5mm,width=2mm]}] (3) to[out=310, in=50, looseness=7] (3);
			\end{tikzpicture}
			\tcbitem \begin{tikzpicture}[scale=.9]
				\tikzset{vertex/.style = {}}
				\tikzset{every loop/.style={min distance=10mm,in=45,out=-45,->}}
				\tikzset{edge/.style={decoration={markings,mark=at position 1 with %
							{\arrow[scale=1.5,>=stealth]{>}}},postaction={decorate}}}
				\node[vertex] (1) at (0, 0) {$\bullet$};
				\node[vertex] (2) at (2, 0) {$\bullet$};
				\node[vertex] (3) at (4, 0) {$\bullet$};
				\draw[-{Latex[length=1.5mm,width=2mm]}] (1) to (2);
				\draw[-{Latex[length=1.5mm,width=2mm]}] (2) to (3);
				\draw[-{Latex[length=1.5mm,width=2mm]}] (3) to[out=310, in=50, looseness=7] (3);
			\end{tikzpicture}
			\tcbitem \begin{tikzpicture}[scale=.9]
				\tikzset{vertex/.style = {}}
				\tikzset{every loop/.style={min distance=10mm,in=45,out=-45,->}}
				\tikzset{edge/.style={decoration={markings,mark=at position 1 with %
							{\arrow[scale=1.5,>=stealth]{>}}},postaction={decorate}}}
				\node[vertex] (1) at (0, 0) {$\bullet$};
				\node[vertex] (2) at (2, 0) {$\bullet$};
				\node[vertex] (3) at (4, 0) {$\bullet$};
				\draw[-{Latex[length=1.5mm,width=2mm]}] (1) to (2);
				\draw[-{Latex[length=1.5mm,width=2mm]}] (2) to[bend right=30] (3);
				\draw[-{Latex[length=1.5mm,width=2mm]}] (3) to[bend right=30] (2);
			\end{tikzpicture}
			\tcbitem \begin{tikzpicture}[scale=.9]
				\tikzset{vertex/.style = {}}
				\tikzset{every loop/.style={min distance=10mm,in=45,out=-45,->}}
				\tikzset{edge/.style={decoration={markings,mark=at position 1 with %
							{\arrow[scale=1.5,>=stealth]{>}}},postaction={decorate}}}
				\node[vertex] (1) at (2, 0) {$\bullet$};
				\node[vertex] (1-1a) at (.166, .5) {$\bullet$};
				\node[vertex] (1-1b) at (.166, -.5) {$\bullet$};
				\draw[-{Latex[length=1.5mm,width=2mm]}] (1) to[out=310, in=50, looseness=7] (1);
				\draw[-{Latex[length=1.5mm,width=2mm]}] (1-1a) to (1);
				\draw[-{Latex[length=1.5mm,width=2mm]}] (1-1b) to (1);
			\end{tikzpicture}
			\tcbitem \begin{tikzpicture}[scale=.8]
				\tikzset{vertex/.style = {}}
				\tikzset{every loop/.style={min distance=10mm,in=45,out=-45,->}}
				\tikzset{edge/.style={decoration={markings,mark=at position 1 with %
							{\arrow[scale=1.5,>=stealth]{>}}},postaction={decorate}}}
				\node[vertex] (1) at (0, 0) {$\bullet$};
				\node[vertex] (2) at (2, 0) {$\bullet$};
				\node[vertex] (3) at (4, 0) {$\bullet$};
				\node[vertex] (4) at (6, 0) {$\bullet$};
				\draw[-{Latex[length=1.5mm,width=2mm]}] (1) to (2);
				\draw[-{Latex[length=1.5mm,width=2mm]}] (2) to[out=310, in=50, looseness=7] (2);
				\draw[-{Latex[length=1.5mm,width=2mm]}] (3) to (4);
				\draw[-{Latex[length=1.5mm,width=2mm]}] (4) to[out=310, in=50, looseness=7] (4);
			\end{tikzpicture}
			\tcbitem \begin{tikzpicture}[scale=.8]
				\tikzset{vertex/.style = {}}
				\tikzset{every loop/.style={min distance=10mm,in=45,out=-45,->}}
				\tikzset{edge/.style={decoration={markings,mark=at position 1 with %
							{\arrow[scale=1.5,>=stealth]{>}}},postaction={decorate}}}
				\node[vertex] (1) at (0, 0) {$\bullet$};
				\node[vertex] (2) at (2, 0) {$\bullet$};
				\node[vertex] (3) at (4, 0) {$\bullet$};
				\node[vertex] (4) at (6, 0) {$\bullet$};
				\draw[-{Latex[length=1.5mm,width=2mm]}] (1) to (2);
				\draw[-{Latex[length=1.5mm,width=2mm]}] (2) to[out=310, in=50, looseness=7] (2);
				\draw[-{Latex[length=1.5mm,width=2mm]}] (3) to[out=310, in=50, looseness=7] (3);
				\draw[-{Latex[length=1.5mm,width=2mm]}] (4) to[out=310, in=50, looseness=7] (4);
			\end{tikzpicture}
			\tcbitem \begin{tikzpicture}[scale=.8]
				\tikzset{vertex/.style = {}}
				\tikzset{every loop/.style={min distance=10mm,in=45,out=-45,->}}
				\tikzset{edge/.style={decoration={markings,mark=at position 1 with %
							{\arrow[scale=1.5,>=stealth]{>}}},postaction={decorate}}}
				\node[vertex] (1) at (0, 0) {$\bullet$};
				\node[vertex] (2) at (2, 0) {$\bullet$};
				\node[vertex] (3) at (4, 0) {$\bullet$};
				\node[vertex] (4) at (6, 0) {$\bullet$};
				\draw[-{Latex[length=1.5mm,width=2mm]}] (1) to (2);
				\draw[-{Latex[length=1.5mm,width=2mm]}] (2) to (3);
				\draw[-{Latex[length=1.5mm,width=2mm]}] (3) to[bend right=30] (4);
				\draw[-{Latex[length=1.5mm,width=2mm]}] (4) to[bend right=30] (3);
			\end{tikzpicture}
			\tcbitem \begin{tikzpicture}[scale=.9]
				\tikzset{vertex/.style = {}}
				\tikzset{every loop/.style={min distance=10mm,in=45,out=-45,->}}
				\tikzset{edge/.style={decoration={markings,mark=at position 1 with %
							{\arrow[scale=1.5,>=stealth]{>}}},postaction={decorate}}}
				\node[vertex] (1) at (0, 0) {$\bullet$};
				\node[vertex] (2) at (2, 0) {$\bullet$};
				\node[vertex] (3) at (4, 0) {$\bullet$};
				\node[vertex] (4) at (6, 0) {$\bullet$};
				\draw[-{Latex[length=1.5mm,width=2mm]}] (1) to (2);
				\draw[-{Latex[length=1.5mm,width=2mm]}] (2) to[bend right=30] (3);
				\draw[-{Latex[length=1.5mm,width=2mm]}] (3) to[bend right=30] (2);
				\draw[-{Latex[length=1.5mm,width=2mm]}] (4) to (3);
			\end{tikzpicture}
		\end{tcbitemize}
	\end{thm}
	
	\subsection*{Organization of the paper}
	
	In Section \ref{preliminary}, we review basic definitions and notation of height functions. Section \ref{bound} is devoted to the proof of Theorem \ref{main thm}. In Section \ref{application}, we apply this result to prove Theorem \ref{-2 to 2}. Finally, in Section \ref{portrait}, we prove Theorem \ref{portrait thm}.
	
	\subsection*{Acknowledgments}
	
	We thank our supervisor, Yohsuke Matsuzawa, for their invaluable comments and insightful discussions. We also thank Kaoru Sano and Hayato Matsumoto for their helpful advice.
	
	\section{Preliminaries\label{preliminary}}
	
	This section provides a brief overview of height functions, primarily based on \cite[Chapter 3]{sil}. Define the set $M_\mathbb{Q}$ as
	\begin{align*}
		M_\mathbb{Q}
		=
		\{ |\cdot|_p \mid p \text{ is a prime number, or } p=\infty \},
	\end{align*}
	where $|\cdot|_\infty$ denotes the standard archimedean absolute value, and for each prime number $p$, $|\cdot|_p$ denotes the non-archimedean absolute value defined by
	\begin{align*}
		\left|\frac{a}{b}\right|_p
		=
		p^{-\left(v_p(a)-v_p(b)\right)},
	\end{align*}
	where $a,b\in\mathbb Z$, $b\neq0$, and $v_p(n)$ denotes the exponent of $p$ in the prime factorization of a nonzero integer $n$.

	Let $K$ be a number field. For each absolute value $|\cdot|_p \in M_{\mathbb{Q}}$, we consider all absolute values $|\cdot|_v$ on $K$ whose restriction to $\mathbb{Q}$ coincides with $|\cdot|_p$. Let $M_K$ denote the set of all such absolute values on $K$.
	
	\begin{dfn}
		Let $N$ be a positive integer, and $K$ a number field. The \emph{multiplicative height} $H_K(P)$ and the \emph{additive height} $h_K(P)$ of $P=(x_0:\dots:x_N) \in \mathbb{P}^N(K)$ relative to $K$ are defined as
		\begin{align*}
			H_K(P) = \prod_{v \in M_K} \max \{ |x_0|_v, \dots, |x_N|_v \}^{n_v}, \quad
			h_K(P) = \log H_K(P),
		\end{align*}
		where $n_v = [K_v : \mathbb{Q}_p]$ with $p$ being the place of $\mathbb{Q}$ that $v$ lies above.
	\end{dfn}
	
	This definition actually does not depend on the choice of homogeneous coordinates for $P$.
	
	\begin{dfn}
		\label{naive_height}
		Take $P=(x_0:\dots:x_N) \in \mathbb{P}^N(\overline{\mathbb{Q}})$. We define the \emph{absolute multiplicative height} $H(P)$, and the \emph{absolute additive height} or \emph{naive height} $h(P)$ of $P$ as
		\begin{align*}
			H(P) = H_K(P)^{1/[K:\mathbb{Q}]}, \quad
			h(P) = \log H(P),
		\end{align*}
		where $K$ is a number field satisfying $P \in \mathbb{P}^N(K)$. This definition is known to be independent of the choice of $K$.
	\end{dfn}
	
	Although these heights depend on the dimension $N$, we suppress this dependence in the notation as is customary.
	
	\begin{rem}
		Let $x_0, \, \ldots, \, x_N \in \mathbb{Z}$ be coprime integers. Then, we have
		\begin{align}
			H(x_0: \cdots: x_N) = \max\{|x_0|, \, \ldots, \, |x_N|\}. 
		\end{align}
		For a proof, see, for example, {\cite[Remark 3.5]{sil}}.
	\end{rem}
	
	\section{Upper bound of $|\hat{h}_\phi(P) - h(P)|$\label{bound}}
	
	First, we cite a lemma from \cite[Theorem 3.11]{sil}. Although it holds in a more general setting, we present a version restricted to the one-dimensional case.
	
	\begin{lem}
		\label{upper bound}
		Let $\phi: \mathbb{P}^1 \to \mathbb{P}^1$ be a rational map defined over a number field $K$ of degree $d \geq 2$. Then
		\begin{align*}
			h(\phi(P)) - dh(P) \leq \log (d + 1) + h(\phi)
		\end{align*}
		for all $P \in \mathbb{P}^1(\overline{K})$.
	\end{lem}
	
	In order to give a lower bound, we prepare a lemma:
	
	\begin{lem}
		\label{trianglish_matrix}
		Let $C$ be a square matrix of order $d$ of the form
		\begin{equation}
			C =
			\begin{pmatrix}
				a_{11} & a_{12} & 0 & 0 & \cdots & 0 \\
				a_{21} & a_{22} & a_{23} & 0 & \cdots & 0 \\
				\vdots & \vdots & \ddots & \ddots & \ddots & \vdots \\
				a_{d-2,1} & a_{d-2,2} & a_{d-2,3} & a_{d-2,4} & \ddots & 0 \\
				a_{d-1,1} & a_{d-1,2} & a_{d-1,3} & a_{d-1,4} & \ddots & a_{d-1,d} \\
				a_{d1} & a_{d2} & a_{d3} & a_{d4} & \cdots & a_{dd}
			\end{pmatrix},
		\end{equation}
		where each $a_{ij}$ is an element of some ring $R$. Then, $\det C$ is given by a sum of at most $2^{d-1}$ monomials in $a_{ij}$ of degree $d$, each having a coefficient $\pm 1$.
	\end{lem}
	
	\begin{proof}
		We prove the assertion by induction on $d$. For $d = 1$, we have $C = (a_{11})$ and its determinant $a_{11}$. For $d > 1$, by cofactor expansion along the first row, we have
		\begin{align}
			\det C &= a_{11}
			\begin{vmatrix}
				a_{22} & a_{23} & 0 & 0 & \cdots & 0 \\
				a_{32} & a_{33} & a_{34} & 0 & \cdots & 0 \\
				\vdots & \vdots & \ddots & \ddots & \ddots & \vdots \\
				a_{d-2,2} & a_{d-2,3} & a_{d-2,4} & a_{d-2,5} & \ddots & 0 \\
				a_{d-1,2} & a_{d-1,3} & a_{d-1,4} & a_{d-1,5} & \ddots & a_{d-1,d} \\
				a_{d2} & a_{d3} & a_{d4} & a_{d5} & \cdots & a_{dd}
			\end{vmatrix} \\
			&- a_{12}
			\begin{vmatrix}
				a_{21} & a_{23} & 0 & 0 & \cdots & 0 \\
				a_{31} & a_{33} & a_{34} & 0 & \cdots & 0 \\
				\vdots & \vdots & \ddots & \ddots & \ddots & \vdots \\
				a_{d-2,1} & a_{d-2,3} & a_{d-2,4} & a_{d-2,5} & \ddots & 0 \\
				a_{d-1,1} & a_{d-1,3} & a_{d-1,4} & a_{d-1,5} & \ddots & a_{d-1,d} \\
				a_{d1} & a_{d3} & a_{d4} & a_{d5} & \cdots & a_{dd}
			\end{vmatrix}.
		\end{align}
		By induction hypothesis, the assertion follows.
	\end{proof}
	
	Next, we give a lower bound:
	
	\begin{lem}
		\label{lower bound}
		Let $\phi: \mathbb{P}^1 \to \mathbb{P}^1$ be a polynomial map defined over a number field $K$ of degree $d \geq 2$. Then
		\begin{align*}
			h(\phi(P)) - dh(P) \geq - \big( d \log 2 + \log d + dh(\phi) \big)
		\end{align*}
		for all $P \in \mathbb{P}^1(\overline{K})$.
	\end{lem}
	
	\begin{proof}
		Let $\phi = (f_0(X_0,X_1): f_1(X_0,X_1))$ with homogeneous polynomials
		\begin{equation}
			f_0(X_0,X_1) = a_0 X_0^d + a_1 X_0^{d-1} X_1 + \dots + a_d X_1^d, \quad f_1(X_0,X_1) \in K[X_0,X_1]
		\end{equation}
		of degree $d$. Since $\phi$ is a polynomial map, we can take
		\begin{equation}
			f_1(X_0,X_1) = X_1^d.
		\end{equation}
		Following the proof of \cite[Proposition 3.2]{yap}, we will find $u_{i0}$, \ldots, $u_{i, d-1}$, $v_{i0}$, \ldots, $v_{i, d-1} \in K$ for $i = 0, \, 1$ such that
		\begin{equation}
			\label{equation}
			X_i^{2d-1} = g_{i0}(X_0, X_1) f_0(X_0, X_1) + g_{i1}(X_0, X_1) X_1^d,
		\end{equation}
		where
		\begin{align}
			&g_{i0}(X_0,X_1) = u_{i0} X_0^{d-1} + u_{i1} X_0^{d-2} X_1 + \dots + u_{i, d-1}X_1^{d-1}, \\
			&g_{i1}(X_0,X_1) = v_{i0} X_0^{d-1} + v_{i1} X_0^{d-2} X_1 + \dots + v_{i, d-1}X_1^{d-1}.
		\end{align}
		Since $f_0$ and $X_1^d$ have no common roots other than $(0,0)$, their homogeneous resultant $\det A$ is non-zero, where $A$ is the Sylvester matrix given by
		\[
		A =
		\begin{pmatrix}
			a_0 & 0 & \cdots & 0 & 0 & 0 & \cdots & 0 \\
			a_1 & a_0 & 0 & \vdots & 0 & 0 & \cdots & 0 \\
			a_2 & a_1 & a_0 & \vdots & \vdots & \vdots & \ddots & \vdots \\
			\vdots & \vdots & \vdots & a_0 & 0 & 0 & \cdots & 0 \\
			a_d & a_{d-1} & a_{d-2} & a_1 & 1 & 0 & \cdots & 0 \\
			0 & a_d & a_{d-1} & a_2 & 0 & 1 & 0 & \vdots \\
			\vdots & 0 & \vdots & \vdots & \vdots & 0 & \ddots & 0 \\
			0 & \cdots & 0 & a_d & 0 & \cdots & 0 & 1
		\end{pmatrix}.
		\]
		Let $\delta_{ij}$ denote the Kronecker delta. Consider the system of equations
		\begin{align}
			A
			\begin{pmatrix}
				u_{i0} \\ u_{i1} \\ \vdots \\ u_{i, d - 1} \\ v_{i0} \\ v_{i1} \\ \vdots \\ v_{i, d - 1}
			\end{pmatrix}
			=
			\begin{pmatrix}
				\delta_{i0} \\ 0 \\ \vdots \\ 0 \\ 0 \\ \vdots \\ 0 \\ \delta_{i1}
			\end{pmatrix}.
		\end{align}
		Let $A_{ij}$ be a matrix obtained by replacing the $j$-th column of $A$ with the vector ${}^t (\delta_{i0}, 0, \dots, 0, \delta_{i1})$.
		Then, by Cramer's rule, each $u_{ij}$ and $v_{ij}$ can be expressed as $\frac{\det A_{ik}}{\det A}$ for some $k$. Then, we have
		\begin{align}
			\label{height_of_coeff}
			H(&u_{00}: \dots: u_{0, d - 1}: v_{00}: \dots: v_{0, d - 1}: \\
			&u_{10}: \dots: u_{1, d - 1}: v_{10}: \dots: v_{1, d - 1}) \\
			&= H(\det A_{01}: \dots: \det A_{0,2d}: \det A_{11}: \dots: \det A_{1,2d}).
		\end{align}
		Considering the case where $i=1$, we have $\det A_{1, 2d} = \det A = a_0^d$. For $1 \leq j < 2d$, we have $\det A_{1j} = 0$. For the case where $i=0$, if $1 \leq j \leq d$, it suffices to focus on the determinant of the upper-left submatrix of $A_{0j}$. Thus
		\begin{align}
			\det A_{0j} &=
			\begin{vmatrix}
				a_0 & 0 & \cdots & 1 & \cdots & 0 \\
				a_1 & a_0 & \cdots & 0 & \cdots & 0 \\
				\vdots & \vdots & \cdots & \vdots & \cdots & \vdots \\
				a_{d - 1} & a_{d - 2} & \cdots & 0 & \cdots & a_0
			\end{vmatrix} \\
			&= (-1)^{1 + j}
			\begin{vmatrix}
				a_1 & a_0 & 0 & \cdots & 0 \\
				a_2 & a_1 & a_0 & \cdots & \vdots \\
				a_3 & a_2 & a_1 & \ddots & 0 \\
				\vdots & \vdots & \vdots & \ddots & a_0 \\
				a_{j - 1} & a_{j - 2} & \cdots & a_2 & a_1
			\end{vmatrix}
			\begin{vmatrix}
				a_0 & 0 & \cdots & 0 \\
				a_1 & a_0 & \cdots & 0 \\
				\vdots & \vdots & \ddots & \vdots \\
				a_{d - j - 1} & a_{d - j - 2} & \cdots & a_0
			\end{vmatrix} \\
			&\label{determinant1}= (-1)^{1 + j} a_0^{d - j}
			\begin{vmatrix}
				a_1 & a_0 & 0 & \cdots & 0 \\
				a_2 & a_1 & a_0 & \cdots & \vdots \\
				a_3 & a_2 & a_1 & \ddots & 0 \\
				\vdots & \vdots & \vdots & \ddots & a_0 \\
				a_{j - 1} & a_{j - 2} & \cdots & a_2 & a_1
			\end{vmatrix}.
		\end{align}
		If $d + 1 \leq j \leq 2d$, by cofactor expansion along the $(d+1)$-th to $2d$-th columns, we have
		\begin{align}
			\label{determinant2}
			\det A_{0j} &= \pm
			\begin{vmatrix}
				a_1 & a_0 & \cdots & 0 & \cdots & 0 & 0 \\
				a_2 & a_1 & \ddots & \vdots & \ddots & \ddots & \vdots \\
				a_3 & a_2 & \ddots & a_0 & \ddots & 0 & 0 \\
				\vdots & \vdots & \vdots & a_1 & \ddots & 0 & 0 \\
				a_{d - 2} & a_{d - 3} & \cdots & \vdots & \ddots & a_0 & 0 \\
				a_{d - 1} & a_{d - 2} & \cdots & a_{2d - j} & \cdots & a_1 & a_0 \\
				0 & 0 & \cdots & a_d & \cdots & a_{j - d + 1} & a_{j - d}
			\end{vmatrix}.
		\end{align}
		For any integer $m$ and $v \in M_K$, we define $\delta_v(m)$ to be $m$ if $v$ is archimedean and $1$ otherwise. Put $|\phi|_v = \max\{|a_0|_v, \, \ldots, \, |a_d|_v, \, 1\}$ and $H(\phi) = H(a_0: \dots : a_d: 1)$ = $\prod_{v \in M_K} |\phi|_v^{n_v}$.
		Then, by \eqref{height_of_coeff}, \eqref{determinant1}, \eqref{determinant2}, and Lemma \ref{trianglish_matrix}, we have
		\begin{align}
			& H(u_{00}: \dots: v_{1, d - 1})
			= H(\det A_{01}: \dots: \det A_{1,2d}) \\
			&= \prod_{v \in M_K} \max_{\substack{i=0,1 \\ 1 \leq j \leq 2d}} |\det A_{ij}|_v^{n_v}
			\leq \prod_{v \in M_K} \big( \delta_v(2^{d-1}) |\phi|_v^d \big)^{n_v} \\
			&= 2^{d-1} \left( \prod_{v \in M_K} |\phi|_v^{n_v} \right)^d
			= 2^{d-1} H(\phi)^d.
		\end{align}
		
		Now, let $P = (x_0: x_1) \in \P^1(\overline{K})$. Take a finite extension $L/K$ so that $P \in \P^1(L)$, and take any $v \in M_L$. 
		Put $|P|_v = \max_{i=0,1} |x_i|_v$. Note that although $|P|_v$ depends on the choice of homogeneous coordinates, $H(P)$ does not.
		By \eqref{equation}, we have
		\begin{align}
			|P|_v^{2d-1} &= \max_{i = 0, 1} |x_i|_v^{2d-1}
			= \max_{i = 0, 1} |g_{i0}(P) f_0(P) + g_{i1}(P) f_1(P)|_v \\
			& \leq \delta_v(2) \max_{\substack{i = 0, 1 \\ j = 0, 1}} |g_{ij}(P) f_j(P)|_v 
			\leq \delta_v(2) \max_{\substack{i = 0, 1 \\ j = 0, 1}} \delta_v(d) |g_{ij}|_v |P|_v^{d - 1} |f_j(P)|_v \\
			& \leq \delta_v(2d) \left( \max_{\substack{i = 0, 1 \\ j = 0, 1}} |g_{ij}|_v \right) |P|_v^{d - 1} |\phi(P)|_v.
		\end{align}
		Hence, we have
		\begin{equation}
			|P|_v^d \leq \delta_v(2d) \left( \max_{\substack{i = 0, 1 \\ j = 0, 1}} |g_{ij}|_v \right) |\phi(P)|_v.
		\end{equation}
		Raising both sides to the power of $n_v$ and take the product over $v \in M_L$, we have
		\begin{align}
			H(P)^d &= \left( \prod_{v \in M_L} |P|_v^{n_v} \right)^d
			= \prod_{v \in M_L} \left( |P|_v^d \right)^{n_v} \\
			&\leq \prod_{v \in M_L} \left( \delta_v(2d) \left( \max_{\substack{i = 0, 1 \\ j = 0, 1}} |g_{ij}|_v \right) |\phi(P)|_v \right)^{n_v} \\
			&= 2d H(u_{00}: \cdots: v_{1, d-1}) H(\phi(P)) \\
			&\leq 2^d d H(\phi)^d H(\phi(P)).
		\end{align}
		Taking the logarithm of both sides, we have
		\begin{equation}
			h(\phi(P)) - dh(P) \geq - \big( d\log2 + \log d + dh(\phi) \big).
		\end{equation}
	\end{proof}
	
	Now, we prove Theorem \ref{main thm}.
	
	\begin{proof}[Proof of Theorem \ref{main thm}]
		Let $P \in \P^1(\overline{K})$.
		Combining Lemma \ref{upper bound} with Lemma \ref{lower bound}, we have
		\begin{equation}
			| h(\phi(P)) - dh(P) | \leq d \log 2 + \log d + dh(\phi).
		\end{equation}
		Hence
		\begin{align}
			\left| \frac{h(\phi^n(P))}{d^n} - h(P) \right| &\leq
			\sum_{i = 1}^{n} \left| \frac{h(\phi^i(P))}{d^i} - \frac{h(\phi^{i - 1}(P))}{d^{i - 1}} \right| \\
			&\leq \sum_{i = 1}^{n} \frac{1}{d^i} \big( d\log2 + \log d + dh(\phi) \big) \\
			&= \frac{1 - (1/d)^n}{1 - 1/d} \left( \log2 + \frac{\log d}{d} + h(\phi) \right).
		\end{align}
		Taking the limit as $n \to \infty$, we obtain the assertion of the theorem.
	\end{proof}
	
	\section{Proof of Theorem \ref{-2 to 2}\label{application}}
	
	Throughout this sectoin and the next, we work over $\mathbb{Q}$.
	
	\begin{dfn}
		Let $\phi \colon \P^1 \to \P^1$ be a rational map defined over $\mathbb{Q}$.
		A point $P \in \P^1(\Q)$ is \emph{preperiodic} if there exist integers $n \geq 0$ and $m \geq 1$ such that $\phi^{n+m}(P) = \phi^n(P)$.
	\end{dfn}
	
	\begin{dfn}
		For $\alpha \in \mathbb{Q}$, $h(\alpha)$ is defined to be $h((\alpha:1))$.
	\end{dfn}
	
	From now on, we apply Theorem \ref{main thm} to polynomial maps $\phi$ of height $0$.
	In this case, we can write $\phi$ as
	\begin{equation}
		\phi(X) = a_d X^d + \cdots + a_0, \quad a_d, \, \ldots, \, a_0 \in \{-1, \, 0, \, 1\} \text{ with } a_d \neq 0.
	\end{equation}
	We investigate which rational numbers can be preperiodic for some such polynomials. Recall:
	
	\begin{lem}[e.g.~{\cite[Theorem 3.22]{sil}}]
		Let $\phi \colon \P^1 \to \P^1$ be a rational map defined over $\Q$.
		Then, for any preperiodic point $P \in \P^1(\overline{\Q})$ for $\phi$, we have $\hat{h}_\phi (P) = 0$.
	\end{lem}
	
	Therefore, if a rational number $\alpha$ is preperiodic for $\phi$ of height $0$, it is necessary that we have
	\begin{equation}
		\label{inequality}
		|h(\alpha)| \leq \frac{1}{1 - 1/d} \left( \log2 + \frac{\log d}{d} \right).
	\end{equation}
	
	Moreover, we prepare the following lemma:
	
	\begin{lem}
		\label{sanosan}
		Let $\phi(X) = \pm X^d + a_{d-1} X^{d-1} + \cdots + a_0 \in \mathbb{Z}[X]$ be a polynomial map. If a rational number $\alpha \in \Q$ is preperiodic for $\phi$, then $\alpha \in \mathbb{Z}$.
	\end{lem}
	
	\begin{proof}
		By the assumption that $\alpha$ is preperiodic for $\phi$, we have $\phi^{n + m}(\alpha) - \phi^n(\alpha) = 0$ for some $n \geq 0$ and $m > 0$. Since the coefficients of $\phi^{n + m}(X) - \phi^n(X)$ are integral and its leading coefficient is $\pm 1$, $\alpha$ is an algebraic integer. Given that $\alpha \in \Q$, it follows that $\alpha \in \mathbb{Z}$, since $\mathbb{Z}$ is integrally closed.
	\end{proof}
	
	\subsection{Case $d = 2$}
	
	For $d = 2$, Inequality \eqref{inequality} can be rewritten as
	\begin{equation}
		|H(\alpha)| \leq 8.
	\end{equation}
	This implies that when expressed as an irreducible fraction, both the numerator and the denominator of preperiodic points for a polynomial map of height $0$ and degree $2$ are at most $8$. Furthermore, by Lemma \ref{sanosan}, we do not need to consider $\alpha \notin \mathbb{Z}$.
	
	Let $a, \, b \in \{-1, \, 0, \, 1\}$. If $\alpha \geq 4$, then
	\begin{equation}
		\alpha^2 + a \alpha + b \geq \alpha^2 - \alpha - 1 \geq 4^2 - 4 - 1 = 11.
	\end{equation}
	Thus, we do not need to consider $\alpha \geq 4$. Similarly, if $\alpha \leq -4$, then
	\begin{equation}
		\alpha^2 + a \alpha + b \geq \alpha^2 + \alpha - 1 \geq (-4)^2 + (-4) - 1 = 11.
	\end{equation}
	Thus, we do not need to consider $\alpha \leq -4$. Next, for $\alpha = \pm 3$, we have
	\begin{align}
		3^2 + 3a + b &\geq 3^2 - 3 - 1 = 5, \\
		(-3)^2 + (-3)a + b &\geq (-3)^2 + (-3) - 1 = 5.
	\end{align}
	Recall that we have already established that integers at least $4$ are not preperiodic. Since the image of $\pm 3$ is at least $5$ (which implies it maps to a non-preperiodic point), it follows that $\pm 3$ are not preperiodic either.
	
	Therefore, if $\alpha$ is a preperiodic point for a quadratic polynomial map of height $0$ defined over $\Q$, then $\alpha$ is an integer satisfying $-2 \leq \alpha \leq 2$.
	
	\subsection{Case $d \geq 3$}
	
	For $d \geq 3$, Inequality \eqref{inequality} can be rewritten as
	\begin{equation}
		|H(\alpha)| \leq 4.
	\end{equation}
	By Lemma \ref{sanosan}, it suffices to consider integers $\alpha$ such that $|\alpha| \leq 4$.
	Furthermore, if $\alpha \geq 3$, then
	\begin{align}
		\alpha^d + a_{d - 1} \alpha^{d - 1} + \cdots + a_0
		&\geq \alpha^d - \alpha^{d - 1} - \cdots - 1
		\geq \alpha^3 - \alpha^2 - \alpha - 1 \\
		&\geq 3^3 - 3^2 - 3 - 1 = 14.
	\end{align}
	Hence, if $d \geq 3$, then $\alpha \geq 3$ is not preperiodic for polynomial maps of height $0$. Next, consider $\alpha \leq -3$. If $d \geq 3$ is odd, we have the inequality
	\begin{align}
		\alpha^d + a_{d - 1} \alpha^{d - 1} + \cdots + a_0
		&\leq \alpha^d + \alpha^{d-1} - \alpha^{d-2} + \cdots - \alpha + 1 = \alpha^d + \frac{\alpha^d + 1}{\alpha + 1}.
	\end{align}
	Let $g_d(\alpha)$ denote the right-hand side. The difference $g_d(\alpha) - g_3(\alpha)$ factors as
	\begin{align}
		g_d(\alpha) - g_3(\alpha) = (\alpha^d - \alpha^3) + \frac{\alpha^d - \alpha^3}{\alpha+1} = \frac{\alpha^3(\alpha^{d-3}-1)(\alpha+2)}{\alpha+1}.
	\end{align}
	Since $\alpha \le -3$ and $d-3$ is an even non-negative integer, this fraction is non-positive, implying $g_d(\alpha) \le g_3(\alpha)$.
	Furthermore,
	\begin{align}
		g_3(\alpha) - (-14) = \alpha^3 + \alpha^2 - \alpha + 15 = (\alpha+3) \big( (\alpha-1)^2+4 \big) \le 0.
	\end{align}
	Thus, we conclude $g_3(\alpha) \leq -14$.
	
	Similarly, if $d \geq 4$ is even, we have
	\begin{align}
		\alpha^d + a_{d - 1} \alpha^{d - 1} + \cdots + a_0
		&\geq \alpha^d + \alpha^{d-1} - \alpha^{d-2} + \cdots + \alpha - 1 = \alpha^d + \frac{\alpha^d - 1}{\alpha + 1}.
	\end{align}
	Let $h_d(\alpha)$ denote the right-hand side. The difference $h_d(\alpha) - h_4(\alpha)$ factors as
	\begin{align}
		h_d(\alpha) - h_4(\alpha) = (\alpha^d - \alpha^4) + \frac{\alpha^d - \alpha^4}{\alpha+1} = \frac{\alpha^4(\alpha^{d-4}-1)(\alpha+2)}{\alpha+1}.
	\end{align}
	Since $\alpha \le -3$ and $d-4$ is an even non-negative integer, this fraction is non-negative, implying $h_d(\alpha) \ge h_4(\alpha)$.
	Moreover,
	\begin{align}
		h_4(\alpha) - 41 = \alpha^4 + \alpha^3 - \alpha^2 + \alpha - 42 = (\alpha+3)(\alpha^3 - 2\alpha^2 + 5\alpha - 14) \ge 0,
	\end{align}
	where the last inequality holds because $\alpha+3 \le 0$ and the cubic polynomial $\alpha^3 - 2\alpha^2 + 5\alpha - 14$ is strictly negative (as all terms are negative for $\alpha \le -3$).
	Hence, $|\alpha| \geq 3$ is not preperiodic for polynomial maps of height $0$. This completes the proof of Theorem \ref{-2 to 2}.
	
	\section{Proof of Theorem \ref{portrait thm}\label{portrait}}
	
	In this section, we determine which portrait can be realized. As in the previous section, we work over 
	$\mathbb{Q}$. First, it is known that all periodic points under our consideration have a period of at most 2.
	
	\begin{lem}
		Let $\phi(X) = a_d X^d + \cdots + a_0 \in \mathbb{Z}[X]$ where $a_d = \pm 1$.
		Then, all integral periodic points for $\phi$ have a period of at most 2.
	\end{lem}
	
	\begin{proof}
		This follows from \cite[Lemma 28]{zieve}.
	\end{proof}
	
	According to this lemma, the portrait shown below, for example, cannot be realized:
	
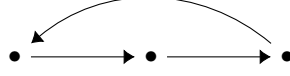
\begin{equation}
		\begin{tikzpicture}[scale=.9]
			\tikzset{vertex/.style = {}}
			\tikzset{every loop/.style={min distance=10mm,in=45,out=-45,->}}
			\tikzset{edge/.style={decoration={markings,mark=at position 1 with %
						{\arrow[scale=1.5,>=stealth]{>}}},postaction={decorate}}}
			\node[vertex] (1) at (0, 0) {$\bullet$};
			\node[vertex] (2) at (2, 0) {$\bullet$};
			\node[vertex] (3) at (4, 0) {$\bullet$};
			\draw[-{Latex[length=1.5mm,width=2mm]}] (1) to (2);
			\draw[-{Latex[length=1.5mm,width=2mm]}] (2) to (3);
			\draw[-{Latex[length=1.5mm,width=2mm]}] (3) to[bend right=40] (1);
		\end{tikzpicture}
	\end{equation}
	
	First, we determine when $2$ or $-2$ is preperiodic for a polynomial map $\phi$ of height $0$.
	
	\subsection{When is $2$ preperiodic}
	
	Let $\phi(X) = \sum_{i=0}^{d} a_i X^i$ with $a_i \in \{-1, \, 0, \, 1\}$ and $a_d \neq 0$. If $\phi$ is monic, it follows that
	\begin{align}
		\phi(2) \geq 2^d - \sum_{i=0}^{d-1} 2^i = 1.
	\end{align}
	If $a_d = -1$, we have
	\begin{align}
		\phi(2) \leq -2^d + \sum_{i=0}^{d-1} 2^i = -1.
	\end{align}
	Therefore, we only need to consider the following cases:
	\begin{enumerate}
		\item For $\phi(X) = X^d - X^{d-1} - \cdots - X - 1$, we have $2 \mapsto 1 \mapsto 1 - d$.
		\begin{itemize}
			\item If $d = 2$, then $\phi(-1) = (-1)^2 - (-1) - 1 = 1$, so $2$ is preperiodic.
			In this case, we have $\phi(-2) = (-2)^2 - (-2) - 1 = 5$, which is not preperiodic. Hence, the portrait is as follows:
			
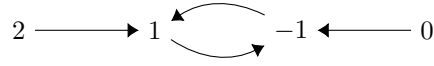
\begin{equation}
				\begin{tikzpicture}[scale=.9]
					\tikzset{vertex/.style = {}}
					\tikzset{every loop/.style={min distance=10mm,in=45,out=-45,->}}
					\tikzset{edge/.style={decoration={markings,mark=at position 1 with %
								{\arrow[scale=1.5,>=stealth]{>}}},postaction={decorate}}}
					\node[vertex] (1) at (0, 0) {$2$};
					\node[vertex] (2) at (2, 0) {$1$};
					\node[vertex] (3) at (4, 0) {$-1$};
					\node[vertex] (4) at (6, 0) {$0$};
					\draw[-{Latex[length=1.5mm,width=2mm]}] (1) to (2);
					\draw[-{Latex[length=1.5mm,width=2mm]}] (2) to[bend right=30] (3);
					\draw[-{Latex[length=1.5mm,width=2mm]}] (3) to[bend right=30] (2);
					\draw[-{Latex[length=1.5mm,width=2mm]}] (4) to (3);
				\end{tikzpicture}
			\end{equation}
			\item If $d = 3$, then $2$ is not preperiodic because
			$\phi(-2) = (-2)^3 - (-2)^2 - (-2) - 1 = -11$.
			\item If $d \geq 4$, then $2$ is not preperiodic since $\phi(1) \leq -3$.
		\end{itemize}
		
		\item For $\phi(X) = X^d - X^{d-1} - \cdots - X$, we have
		\begin{equation}
			\phi(2) = 2^d - \frac{2(2^{d-1} - 1)}{2 - 1} = 2.
		\end{equation}
		Hence, $2$ is periodic for this $\phi$. Since we have $\phi(1) = 2 - d$, $|\phi(-2)| = \frac{1}{3} \left| (-2)^{d+2} + 2 \right| \geq 6$, and
		\begin{align}
			\phi(-1) =
			\begin{cases}
				2 & (d \text{ is even}) \\
				-1 & (d \text{ is odd})
			\end{cases},
		\end{align}
		we have the portrait for each $d$:
		\begin{itemize}
			\item If $d = 2$, we have
			
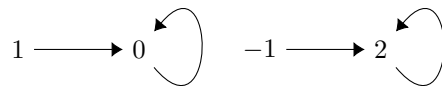
\begin{equation}
				\begin{tikzpicture}[scale=.8]
					\tikzset{vertex/.style = {}}
					\tikzset{every loop/.style={min distance=10mm,in=45,out=-45,->}}
					\tikzset{edge/.style={decoration={markings,mark=at position 1 with %
								{\arrow[scale=1.5,>=stealth]{>}}},postaction={decorate}}}
					\node[vertex] (1) at (0, 0) {$1$};
					\node[vertex] (2) at (2, 0) {$0$};
					\node[vertex] (3) at (4, 0) {$-1$};
					\node[vertex] (4) at (6, 0) {$2$};
					\draw[-{Latex[length=1.5mm,width=2mm]}] (1) to (2);
					\draw[-{Latex[length=1.5mm,width=2mm]}] (2) to[out=310, in=50, looseness=7] (2);
					\draw[-{Latex[length=1.5mm,width=2mm]}] (3) to (4);
					\draw[-{Latex[length=1.5mm,width=2mm]}] (4) to[out=310, in=50, looseness=7] (4);
				\end{tikzpicture}
			\end{equation}
			\item If $d = 3$, we have
			
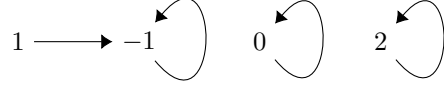
\begin{equation}\begin{tikzpicture}[scale=.8]
					\tikzset{vertex/.style = {}}
					\tikzset{every loop/.style={min distance=10mm,in=45,out=-45,->}}
					\tikzset{edge/.style={decoration={markings,mark=at position 1 with %
								{\arrow[scale=1.5,>=stealth]{>}}},postaction={decorate}}}
					\node[vertex] (1) at (0, 0) {$1$};
					\node[vertex] (2) at (2, 0) {$-1$};
					\node[vertex] (3) at (4, 0) {$0$};
					\node[vertex] (4) at (6, 0) {$2$};
					\draw[-{Latex[length=1.5mm,width=2mm]}] (1) to (2);
					\draw[-{Latex[length=1.5mm,width=2mm]}] (2) to[out=310, in=50, looseness=7] (2);
					\draw[-{Latex[length=1.5mm,width=2mm]}] (3) to[out=310, in=50, looseness=7] (3);
					\draw[-{Latex[length=1.5mm,width=2mm]}] (4) to[out=310, in=50, looseness=7] (4);
				\end{tikzpicture}
			\end{equation}
			\item If $d \geq 4$ is even, we have
			
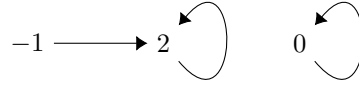
\[
			\begin{tikzpicture}[scale=.9]
				\tikzset{vertex/.style = {}}
				\tikzset{every loop/.style={min distance=10mm,in=45,out=-45,->}}
				\tikzset{edge/.style={decoration={markings,mark=at position 1 with %
							{\arrow[scale=1.5,>=stealth]{>}}},postaction={decorate}}}
				\node[vertex] (1) at (0, 0) {$-1$};
				\node[vertex] (2) at (2, 0) {$2$};
				\node[vertex] (3) at (4, 0) {$0$};
				\draw[-{Latex[length=1.5mm,width=2mm]}] (1) to (2);
				\draw[-{Latex[length=1.5mm,width=2mm]}] (2) to[out=310, in=50, looseness=7] (2);
				\draw[-{Latex[length=1.5mm,width=2mm]}] (3) to[out=310, in=50, looseness=7] (3);
			\end{tikzpicture}
			\]
			\item If $d \geq 5$ is odd, we have
			
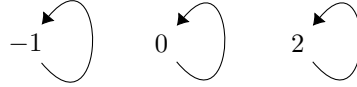
\[
			\begin{tikzpicture}[scale=.9]
				\tikzset{vertex/.style = {}}
				\tikzset{every loop/.style={min distance=10mm,in=45,out=-45,->}}
				\tikzset{edge/.style={decoration={markings,mark=at position 1 with %
							{\arrow[scale=1.5,>=stealth]{>}}},postaction={decorate}}}
				\node[vertex] (1) at (0, 0) {$-1$};
				\node[vertex] (2) at (2, 0) {$0$};
				\node[vertex] (3) at (4, 0) {$2$};
				\draw[-{Latex[length=1.5mm,width=2mm]}] (1) to[out=310, in=50, looseness=7] (1);
				\draw[-{Latex[length=1.5mm,width=2mm]}] (2) to[out=310, in=50, looseness=7] (2);
				\draw[-{Latex[length=1.5mm,width=2mm]}] (3) to[out=310, in=50, looseness=7] (3);
			\end{tikzpicture}
			\]
		\end{itemize}
		
		\item For $\phi(X) = - X^d + X^{d-1} + \cdots + X + 1$, we have
		\begin{equation}
			\phi(2) = -1, \quad
			\phi(-1) = \begin{cases}
				2 & (d \text{ is odd}) \\
				-1 & (d \text{ is even})
			\end{cases}.
		\end{equation}
		Hence, $2$ is preperiodic for this $\phi$. Moreover, we have $0 \mapsto 1$, which can be preperiodic if $d = 2, \, 3$ since $\phi(1) = d-1$. Furthermore, we have
		\begin{align}
			|\phi(-2)| = \frac{1}{3} \left| (-2)^{d+2} - 1 \right| \geq 5 \text{ for } d \geq 2.
		\end{align}
		Hence $-2$ is not preperiodic for each $d \geq 2$. Therefore, we have the portrait for each $d \geq 2$:
		\begin{itemize}
			\item If $d=2$, we have
			
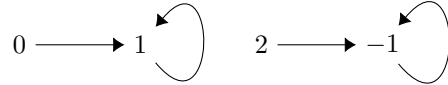
\[
			\begin{tikzpicture}[scale=.8]
				\tikzset{vertex/.style = {}}
				\tikzset{every loop/.style={min distance=10mm,in=45,out=-45,->}}
				\tikzset{edge/.style={decoration={markings,mark=at position 1 with %
							{\arrow[scale=1.5,>=stealth]{>}}},postaction={decorate}}}
				\node[vertex] (1) at (0, 0) {$0$};
				\node[vertex] (2) at (2, 0) {$1$};
				\node[vertex] (3) at (4, 0) {$2$};
				\node[vertex] (4) at (6, 0) {$-1$};
				\draw[-{Latex[length=1.5mm,width=2mm]}] (1) to (2);
				\draw[-{Latex[length=1.5mm,width=2mm]}] (2) to[out=310, in=50, looseness=7] (2);
				\draw[-{Latex[length=1.5mm,width=2mm]}] (3) to (4);
				\draw[-{Latex[length=1.5mm,width=2mm]}] (4) to[out=310, in=50, looseness=7] (4);
			\end{tikzpicture}
			\]
			\item If $d=3$, we have
			
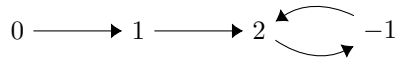
\begin{equation}
				\begin{tikzpicture}[scale=.8]
					\tikzset{vertex/.style = {}}
					\tikzset{every loop/.style={min distance=10mm,in=45,out=-45,->}}
					\tikzset{edge/.style={decoration={markings,mark=at position 1 with %
								{\arrow[scale=1.5,>=stealth]{>}}},postaction={decorate}}}
					\node[vertex] (1) at (0, 0) {$0$};
					\node[vertex] (2) at (2, 0) {$1$};
					\node[vertex] (3) at (4, 0) {$2$};
					\node[vertex] (4) at (6, 0) {$-1$};
					\draw[-{Latex[length=1.5mm,width=2mm]}] (1) to (2);
					\draw[-{Latex[length=1.5mm,width=2mm]}] (2) to (3);
					\draw[-{Latex[length=1.5mm,width=2mm]}] (3) to[bend right=30] (4);
					\draw[-{Latex[length=1.5mm,width=2mm]}] (4) to[bend right=30] (3);
				\end{tikzpicture}
			\end{equation}
			\item If $d \geq 4$ is even, we have
			
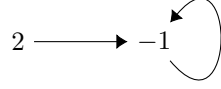
\[
			\begin{tikzpicture}[scale=.9]
				\tikzset{vertex/.style = {}}
				\tikzset{every loop/.style={min distance=10mm,in=45,out=-45,->}}
				\tikzset{edge/.style={decoration={markings,mark=at position 1 with %
							{\arrow[scale=1.5,>=stealth]{>}}},postaction={decorate}}}
				\node[vertex] (1) at (0, 0) {$2$};
				\node[vertex] (2) at (2, 0) {$-1$};
				\draw[-{Latex[length=1.5mm,width=2mm]}] (1) to (2);
				\draw[-{Latex[length=1.5mm,width=2mm]}] (2) to[out=310, in=50, looseness=7] (2);
			\end{tikzpicture}
			\]
			\item If $d \geq 5$ is odd, we have
			
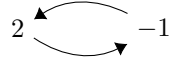
\[
			\begin{tikzpicture}[scale=.9]
				\tikzset{vertex/.style = {}}
				\tikzset{every loop/.style={min distance=10mm,in=45,out=-45,->}}
				\tikzset{edge/.style={decoration={markings,mark=at position 1 with %
							{\arrow[scale=1.5,>=stealth]{>}}},postaction={decorate}}}
				\node[vertex] (1) at (0, 0) {$2$};
				\node[vertex] (2) at (2, 0) {$-1$};
				\draw[-{Latex[length=1.5mm,width=2mm]}] (1) to[bend right=30] (2);
				\draw[-{Latex[length=1.5mm,width=2mm]}] (2) to[bend right=30] (1);
			\end{tikzpicture}
			\]
		\end{itemize}
		
		\item For $\phi(X) = - X^d + X^{d-1} + \cdots + X$, we have $\phi(2) = -2$, and
		\begin{align}
			|\phi(-2)| = \left| -(-2)^d + \frac{-2 \big( 1 - (-2)^{d-1} \big)}{3} \right| = \frac{1}{3} \big| (-2)^{d+2} + 2 \big| \geq 6
		\end{align}
		for each $d \geq 2$. Hence, $2$ is not preperiodic.
	\end{enumerate}
	
	\subsection{When is $-2$ preperiodic}
	
	Let $\phi(X) = a_d X^d + \cdots + a_0$ be a polynomial map of height $0$ for which $-2$ is preperiodic. Define $\psi$ to make the following diagram commutative:
	\begin{equation}
		\begin{tikzcd}
			\Q \arrow[r, "\psi"] \arrow[d, "-1"'] & \Q \\
			\Q \arrow[r, "\phi"'] & \Q \arrow[u, "-1"']
		\end{tikzcd}
	\end{equation}
	Then, $\psi$ is a polynomial of height $0$, and $2$ is preperiodic under $\psi$. The portrait of $\psi$ is isomorphic to that of $\phi$.
	
	From the above, we have completely determined the portraits for the cases where $2$ or $-2$ is preperiodic.
	
	\subsection{When neither $2$ nor $-2$ is preperiodic}
	
	Finally, we determine the portraits for the cases where neither $2$ nor $-2$ is preperiodic. First, suppose there is exactly one preperiodic point. In this case, the portrait is given as follows:
	
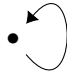
\begin{center}
		\begin{tikzpicture}[scale=.9]
			\tikzset{vertex/.style = {}}
			\tikzset{every loop/.style={min distance=10mm,in=45,out=-45,->}}
			\tikzset{edge/.style={decoration={markings,mark=at position 1 with %
						{\arrow[scale=1.5,>=stealth]{>}}},postaction={decorate}}}
			\node[vertex] (1) at (0, 0) {$\bullet$};
			\draw[-{Latex[length=1.5mm,width=2mm]}] (1) to[out=310, in=50, looseness=7] (1);
		\end{tikzpicture}
	\end{center}
	This portrait is realized, for example, by $\phi(X) = X^3 + X$, since $0 \mapsto 0, \, 1 \mapsto 2 \mapsto 10, \, -1 \mapsto -2 \mapsto -10$.
	
	Next, suppose there are exactly two preperiodic points. In this case, the portrait is one of the following three types:
	
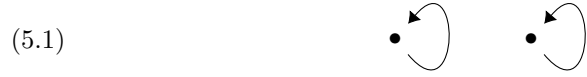
\begin{align}
		\label{2-1}
		\begin{tikzpicture}[scale=.9,baseline={([yshift=-0.8ex]current bounding box.center)}]
			\tikzset{vertex/.style = {}}
			\tikzset{every loop/.style={min distance=10mm,in=45,out=-45,->}}
			\tikzset{edge/.style={decoration={markings,mark=at position 1 with %
						{\arrow[scale=1.5,>=stealth]{>}}},postaction={decorate}}}
			\node[vertex] (1) at (0, 0) {$\bullet$};
			\node[vertex] (2) at (2, 0) {$\bullet$};
			\draw[-{Latex[length=1.5mm,width=2mm]}] (1) to[out=310, in=50, looseness=7] (1);
			\draw[-{Latex[length=1.5mm,width=2mm]}] (2) to[out=310, in=50, looseness=7] (2);
		\end{tikzpicture} \\
		
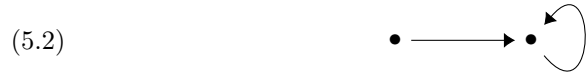
\label{2-2}
		\begin{tikzpicture}[scale=.9,baseline={([yshift=-0.8ex]current bounding box.center)}]
			\tikzset{vertex/.style = {}}
			\tikzset{every loop/.style={min distance=10mm,in=45,out=-45,->}}
			\tikzset{edge/.style={decoration={markings,mark=at position 1 with %
						{\arrow[scale=1.5,>=stealth]{>}}},postaction={decorate}}}
			\node[vertex] (1) at (0, 0) {$\bullet$};
			\node[vertex] (2) at (2, 0) {$\bullet$};
			\draw[-{Latex[length=1.5mm,width=2mm]}] (1) to (2);
			\draw[-{Latex[length=1.5mm,width=2mm]}] (2) to[out=310, in=50, looseness=7] (2);
		\end{tikzpicture} \\
		
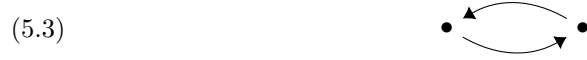
\label{2-3}
		\begin{tikzpicture}[scale=.9,baseline={([yshift=-0.8ex]current bounding box.center)}]
			\tikzset{vertex/.style = {}}
			\tikzset{every loop/.style={min distance=10mm,in=45,out=-45,->}}
			\tikzset{edge/.style={decoration={markings,mark=at position 1 with %
						{\arrow[scale=1.5,>=stealth]{>}}},postaction={decorate}}}
			\node[vertex] (1) at (0, 0) {$\bullet$};
			\node[vertex] (2) at (2, 0) {$\bullet$};
			\draw[-{Latex[length=1.5mm,width=2mm]}] (1) to[bend right=30] (2);
			\draw[-{Latex[length=1.5mm,width=2mm]}] (2) to[bend right=30] (1);
		\end{tikzpicture}
	\end{align}
	The first type \eqref{2-1} is realized, for example, by $\phi(X) = X^3 + X^2 + X$. The second type \eqref{2-2} is realized, for example, by $\phi(X) = X^2 + X$. The third type \eqref{2-3} is realized, for example, by $\phi(X) = -X^3 + 1$.
	
	
	Finally, suppose there are exactly three preperiodic points.
	First, suppose there are exactly three fixed points. In this case, the portrait is given as follows:
	
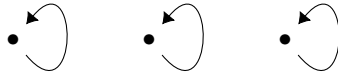
\begin{center}
		\begin{tikzpicture}[scale=.9]
			\tikzset{vertex/.style = {}}
			\tikzset{every loop/.style={min distance=10mm,in=45,out=-45,->}}
			\tikzset{edge/.style={decoration={markings,mark=at position 1 with %
						{\arrow[scale=1.5,>=stealth]{>}}},postaction={decorate}}}
			\node[vertex] (1) at (0, 0) {$\bullet$};
			\node[vertex] (2) at (2, 0) {$\bullet$};
			\node[vertex] (3) at (4, 0) {$\bullet$};
			\draw[-{Latex[length=1.5mm,width=2mm]}] (1) to[out=310, in=50, looseness=7] (1);
			\draw[-{Latex[length=1.5mm,width=2mm]}] (2) to[out=310, in=50, looseness=7] (2);
			\draw[-{Latex[length=1.5mm,width=2mm]}] (3) to[out=310, in=50, looseness=7] (3);
		\end{tikzpicture}
	\end{center}
	This portrait is realized, for example, by $\phi(X) = X^3$.
	
	Next, suppose there are exactly two fixed points. In this case, the portrait is:
	
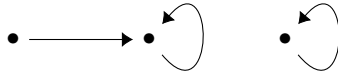
\begin{center}
		\begin{tikzpicture}[scale=.9]
			\tikzset{vertex/.style = {}}
			\tikzset{every loop/.style={min distance=10mm,in=45,out=-45,->}}
			\tikzset{edge/.style={decoration={markings,mark=at position 1 with %
						{\arrow[scale=1.5,>=stealth]{>}}},postaction={decorate}}}
			\node[vertex] (1) at (0, 0) {$\bullet$};
			\node[vertex] (2) at (2, 0) {$\bullet$};
			\node[vertex] (3) at (4, 0) {$\bullet$};
			\draw[-{Latex[length=1.5mm,width=2mm]}] (1) to (2);
			\draw[-{Latex[length=1.5mm,width=2mm]}] (2) to[out=310, in=50, looseness=7] (2);
			\draw[-{Latex[length=1.5mm,width=2mm]}] (3) to[out=310, in=50, looseness=7] (3);
		\end{tikzpicture}
	\end{center}
	This portrait is realized, for example, by $\phi(X) = X^2$.
	
	Next, suppose there is exactly one fixed point. In this case, the portrait is:
	
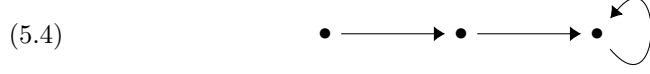
\begin{align}
		\label{3-1-1}
		\begin{tikzpicture}[scale=.9,baseline={([yshift=-0.8ex]current bounding box.center)}]
			\tikzset{vertex/.style = {}}
			\tikzset{every loop/.style={min distance=10mm,in=45,out=-45,->}}
			\tikzset{edge/.style={decoration={markings,mark=at position 1 with %
						{\arrow[scale=1.5,>=stealth]{>}}},postaction={decorate}}}
			\node[vertex] (1) at (0, 0) {$\bullet$};
			\node[vertex] (2) at (2, 0) {$\bullet$};
			\node[vertex] (3) at (4, 0) {$\bullet$};
			\draw[-{Latex[length=1.5mm,width=2mm]}] (1) to (2);
			\draw[-{Latex[length=1.5mm,width=2mm]}] (2) to (3);
			\draw[-{Latex[length=1.5mm,width=2mm]}] (3) to[out=310, in=50, looseness=7] (3);
		\end{tikzpicture} \\
		
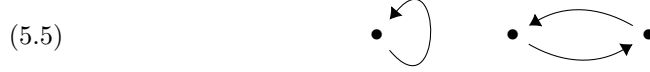
\label{3-1-2}
		\begin{tikzpicture}[scale=.9,baseline={([yshift=-0.8ex]current bounding box.center)}]
			\tikzset{vertex/.style = {}}
			\tikzset{every loop/.style={min distance=10mm,in=45,out=-45,->}}
			\tikzset{edge/.style={decoration={markings,mark=at position 1 with %
						{\arrow[scale=1.5,>=stealth]{>}}},postaction={decorate}}}
			\node[vertex] (1) at (0, 0) {$\bullet$};
			\node[vertex] (2) at (2, 0) {$\bullet$};
			\node[vertex] (3) at (4, 0) {$\bullet$};
			\draw[-{Latex[length=1.5mm,width=2mm]}] (1) to[out=310, in=50, looseness=7] (1);
			\draw[-{Latex[length=1.5mm,width=2mm]}] (2) to[bend right=30] (3);
			\draw[-{Latex[length=1.5mm,width=2mm]}] (3) to[bend right=30] (2);
		\end{tikzpicture} \\
		
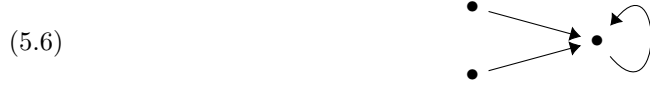
\label{3-1-3}
		\begin{tikzpicture}[scale=.9,baseline={([yshift=-0.8ex]current bounding box.center)}]
			\tikzset{vertex/.style = {}}
			\tikzset{every loop/.style={min distance=10mm,in=45,out=-45,->}}
			\tikzset{edge/.style={decoration={markings,mark=at position 1 with %
						{\arrow[scale=1.5,>=stealth]{>}}},postaction={decorate}}}
			\node[vertex] (1) at (2, 0) {$\bullet$};
			\node[vertex] (1-1a) at (.166, .5) {$\bullet$};
			\node[vertex] (1-1b) at (.166, -.5) {$\bullet$};
			\draw[-{Latex[length=1.5mm,width=2mm]}] (1) to[out=310, in=50, looseness=7] (1);
			\draw[-{Latex[length=1.5mm,width=2mm]}] (1-1a) to (1);
			\draw[-{Latex[length=1.5mm,width=2mm]}] (1-1b) to (1);
		\end{tikzpicture}
	\end{align}
	The first type \eqref{3-1-1} is realized, for example, by $\phi(X) = X^4 + X^2 - 1$.
	The second type \eqref{3-1-2} is realized, for example, by $\phi(X) = -X^3$.
	The third type \eqref{3-1-3} is realized, for example, by $\phi(X) = X^3 - X$.
	
	Finally, suppose there are no fixed points. In this case, the portrait is one of the following two types:
	
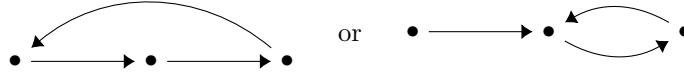
\begin{center}
		\begin{tikzpicture}[scale=.9,baseline={([yshift=-0.8ex]current bounding box.center)}]
			\tikzset{vertex/.style = {}}
			\tikzset{every loop/.style={min distance=10mm,in=45,out=-45,->}}
			\tikzset{edge/.style={decoration={markings,mark=at position 1 with %
						{\arrow[scale=1.5,>=stealth]{>}}},postaction={decorate}}}
			\node[vertex] (1) at (0, 0) {$\bullet$};
			\node[vertex] (2) at (2, 0) {$\bullet$};
			\node[vertex] (3) at (4, 0) {$\bullet$};
			\draw[-{Latex[length=1.5mm,width=2mm]}] (1) to (2);
			\draw[-{Latex[length=1.5mm,width=2mm]}] (2) to (3);
			\draw[-{Latex[length=1.5mm,width=2mm]}] (3) to[bend right=40] (1);
		\end{tikzpicture}
		\quad or \quad
		\begin{tikzpicture}[scale=.9,baseline={([yshift=-0.8ex]current bounding box.center)}]
			\tikzset{vertex/.style = {}}
			\tikzset{every loop/.style={min distance=10mm,in=45,out=-45,->}}
			\tikzset{edge/.style={decoration={markings,mark=at position 1 with %
						{\arrow[scale=1.5,>=stealth]{>}}},postaction={decorate}}}
			\node[vertex] (1) at (0, 0) {$\bullet$};
			\node[vertex] (2) at (2, 0) {$\bullet$};
			\node[vertex] (3) at (4, 0) {$\bullet$};
			\draw[-{Latex[length=1.5mm,width=2mm]}] (1) to (2);
			\draw[-{Latex[length=1.5mm,width=2mm]}] (2) to[bend right=30] (3);
			\draw[-{Latex[length=1.5mm,width=2mm]}] (3) to[bend right=30] (2);
		\end{tikzpicture}
	\end{center}
	The former type does not exist by Lemma 5.1. The latter type is realized, for example, by $\phi(X) = X^2 - 1$.
	
	\printbibliography

@misc{yap,
	title={$S$-Integral Points in Orbits on $\mathbb{P}^1$}, 
	author={Jit Wu Yap},
	year={2023},
	eprint={2312.05094},
	archivePrefix={arXiv},
	primaryClass={math.NT},
	url={https://arxiv.org/abs/2312.05094}, 
}

@book{
	sil,
	title={The arithmetic of dynamical systems},
	author={Silverman, Joseph H.},
	volume={241},
	year={2007},
	publisher={Springer Science \& Business Media}
}

@phdthesis{zieve,
	author={Zieve, Michael},
	title={Cycles of polynomial mappings},
	school={University of California, Berkeley},
	year={1996}
}

@article {hutz,
	AUTHOR = {Hutz, Benjamin},
	TITLE = {Good reduction and canonical heights of subvarieties},
	JOURNAL = {Math. Res. Lett.},
	FJOURNAL = {Mathematical Research Letters},
	VOLUME = {25},
	YEAR = {2018},
	NUMBER = {6},
	PAGES = {1837--1863},
	ISSN = {1073-2780,1945-001X},
	MRCLASS = {37P35 (14G40)},
	MRNUMBER = {3934847},
	MRREVIEWER = {Lukas\ Pottmeyer},
	DOI = {10.4310/MRL.2018.v25.n6.a7},
	URL = {https://doi.org/10.4310/MRL.2018.v25.n6.a7},
}

@article {ingram,
	AUTHOR = {Ingram, Patrick},
	TITLE = {Explicit canonical heights for divisors relative to
	endomorphisms of {$\Bbb P^N$}},
	JOURNAL = {Matematica},
	FJOURNAL = {La Matematica. Official Journal of the Association for Women
	in Mathematics},
	VOLUME = {3},
	YEAR = {2024},
	NUMBER = {4},
	PAGES = {1456--1485},
	ISSN = {2730-9657},
	MRCLASS = {11G50 (14G40 37P05 37P30)},
	MRNUMBER = {4842147},
	MRREVIEWER = {Thomas\ Ward},
	DOI = {10.1007/s44007-024-00100-6},
	URL = {https://doi.org/10.1007/s44007-024-00100-6},
}
\end{document}